\documentclass{amsart} 
\usepackage{amsfonts,amssymb,amsmath,euscript}
    \newtheorem{thm}{Theorem}                     [section]
    \newtheorem{thm*}{Theorem}
    \newtheorem{prop}[thm]{Proposition}
    \newtheorem{lemma}[thm]{Lemma}
    \newtheorem{cor}[thm]{Corollary}

    \newtheorem{lemma*}{Lemma}    

    \newtheorem{assump}[thm]{Assumption}

    \newtheorem{rems*}{Remark}   

\renewenvironment{proof}{{\it Proof.}}{$\Box$\par} 

\newcommand{\ndef}{\newcommand*}
\def\rndef{\renewcommand}

\ndef{\myaddress}[1]{\begin{center} \it\small #1 \end{center}}

\ndef{\clA}{{\mathcal A}} \ndef{\rmA}{{\mathrm A}} \ndef{\mbA}{{\mathbb A}} \ndef{\bfA}{{\mathbf A}} \ndef{\euA}{{\EuScript A}} \ndef{\frA}{{\mathfrak A}}
\ndef{\clB}{{\mathcal B}} \ndef{\rmB}{{\mathrm B}} \ndef{\mbB}{{\mathbb B}} \ndef{\bfB}{{\mathbf B}} \ndef{\euB}{{\EuScript B}} \ndef{\frB}{{\mathfrak B}}
\ndef{\clC}{{\mathcal C}} \ndef{\rmC}{{\mathrm C}} \ndef{\mbC}{{\mathbb C}} \ndef{\bfC}{{\mathbf C}} \ndef{\euC}{{\EuScript C}} \ndef{\frC}{{\mathfrak C}}
\ndef{\clD}{{\mathcal D}} \ndef{\rmD}{{\mathrm D}} \ndef{\mbD}{{\mathbb D}} \ndef{\bfD}{{\mathbf D}} \ndef{\euD}{{\EuScript D}} \ndef{\frD}{{\mathfrak D}}
\ndef{\clE}{{\mathcal E}} \ndef{\rmE}{{\mathrm E}} \ndef{\mbE}{{\mathbb E}} \ndef{\bfE}{{\mathbf E}} \ndef{\euE}{{\EuScript E}} \ndef{\frE}{{\mathfrak E}}
\ndef{\clF}{{\mathcal F}} \ndef{\rmF}{{\mathrm F}} \ndef{\mbF}{{\mathbb F}} \ndef{\bfF}{{\mathbf F}} \ndef{\euF}{{\EuScript F}} \ndef{\frF}{{\mathfrak F}}
\ndef{\clG}{{\mathcal G}} \ndef{\rmG}{{\mathrm G}} \ndef{\mbG}{{\mathbb G}} \ndef{\bfG}{{\mathbf G}} \ndef{\euG}{{\EuScript G}} \ndef{\frG}{{\mathfrak G}}
\ndef{\clH}{{\mathcal H}} \ndef{\rmH}{{\mathrm H}} \ndef{\mbH}{{\mathbb H}} \ndef{\bfH}{{\mathbf H}} \ndef{\euH}{{\EuScript H}} \ndef{\frH}{{\mathfrak H}}
\ndef{\clI}{{\mathcal I}} \ndef{\rmI}{{\mathrm I}} \ndef{\mbI}{{\mathbb I}} \ndef{\bfI}{{\mathbf I}} \ndef{\euI}{{\EuScript I}} \ndef{\frI}{{\mathfrak I}}
\ndef{\clJ}{{\mathcal J}} \ndef{\rmJ}{{\mathrm J}} \ndef{\mbJ}{{\mathbb J}} \ndef{\bfJ}{{\mathbf J}} \ndef{\euJ}{{\EuScript J}} \ndef{\frJ}{{\mathfrak J}}
\ndef{\clK}{{\mathcal K}} \ndef{\rmK}{{\mathrm K}} \ndef{\mbK}{{\mathbb K}} \ndef{\bfK}{{\mathbf K}} \ndef{\euK}{{\EuScript K}} \ndef{\frK}{{\mathfrak K}}
\ndef{\clL}{{\mathcal L}} \ndef{\rmL}{{\mathrm L}} \ndef{\mbL}{{\mathbb L}} \ndef{\bfL}{{\mathbf L}} \ndef{\euL}{{\EuScript L}} \ndef{\frL}{{\mathfrak L}}
\ndef{\clM}{{\mathcal M}} \ndef{\rmM}{{\mathrm M}} \ndef{\mbM}{{\mathbb M}} \ndef{\bfM}{{\mathbf M}} \ndef{\euM}{{\EuScript M}} \ndef{\frM}{{\mathfrak M}}
\ndef{\clN}{{\mathcal N}} \ndef{\rmN}{{\mathrm N}} \ndef{\mbN}{{\mathbb N}} \ndef{\bfN}{{\mathbf N}} \ndef{\euN}{{\EuScript N}} \ndef{\frN}{{\mathfrak N}}
\ndef{\clO}{{\mathcal O}} \ndef{\rmO}{{\mathrm O}} \ndef{\mbO}{{\mathbb O}} \ndef{\bfO}{{\mathbf O}} \ndef{\euO}{{\EuScript O}} \ndef{\frO}{{\mathfrak O}}
\ndef{\clP}{{\mathcal P}} \ndef{\rmP}{{\mathrm P}} \ndef{\mbP}{{\mathbb P}} \ndef{\bfP}{{\mathbf P}} \ndef{\euP}{{\EuScript P}} \ndef{\frP}{{\mathfrak P}}
\ndef{\clQ}{{\mathcal Q}} \ndef{\rmQ}{{\mathrm Q}} \ndef{\mbQ}{{\mathbb Q}} \ndef{\bfQ}{{\mathbf Q}} \ndef{\euQ}{{\EuScript Q}} \ndef{\frQ}{{\mathfrak Q}}
\ndef{\clR}{{\mathcal R}} \ndef{\rmR}{{\mathrm R}} \ndef{\mbR}{{\mathbb R}} \ndef{\bfR}{{\mathbf R}} \ndef{\euR}{{\EuScript R}} \ndef{\frR}{{\mathfrak R}}
\ndef{\clS}{{\mathcal S}} \ndef{\rmS}{{\mathrm S}} \ndef{\mbS}{{\mathbb S}} \ndef{\bfS}{{\mathbf S}} \ndef{\euS}{{\EuScript S}} \ndef{\frS}{{\mathfrak S}}
\ndef{\clT}{{\mathcal T}} \ndef{\rmT}{{\mathrm T}} \ndef{\mbT}{{\mathbb T}} \ndef{\bfT}{{\mathbf T}} \ndef{\euT}{{\EuScript T}} \ndef{\frT}{{\mathfrak T}}
\ndef{\clU}{{\mathcal U}} \ndef{\rmU}{{\mathrm U}} \ndef{\mbU}{{\mathbb U}} \ndef{\bfU}{{\mathbf U}} \ndef{\euU}{{\EuScript U}} \ndef{\frU}{{\mathfrak U}}
\ndef{\clV}{{\mathcal V}} \ndef{\rmV}{{\mathrm V}} \ndef{\mbV}{{\mathbb V}} \ndef{\bfV}{{\mathbf V}} \ndef{\euV}{{\EuScript V}} \ndef{\frV}{{\mathfrak V}}
\ndef{\clW}{{\mathcal W}} \ndef{\rmW}{{\mathrm W}} \ndef{\mbW}{{\mathbb W}} \ndef{\bfW}{{\mathbf W}} \ndef{\euW}{{\EuScript W}} \ndef{\frW}{{\mathfrak W}}
\ndef{\clX}{{\mathcal X}} \ndef{\rmX}{{\mathrm X}} \ndef{\mbX}{{\mathbb X}} \ndef{\bfX}{{\mathbf X}} \ndef{\euX}{{\EuScript X}} \ndef{\frX}{{\mathfrak X}}
\ndef{\clY}{{\mathcal Y}} \ndef{\rmY}{{\mathrm Y}} \ndef{\mbY}{{\mathbb Y}} \ndef{\bfY}{{\mathbf Y}} \ndef{\euY}{{\EuScript Y}} \ndef{\frY}{{\mathfrak Y}}
\ndef{\clZ}{{\mathcal Z}} \ndef{\rmZ}{{\mathrm Z}} \ndef{\mbZ}{{\mathbb Z}} \ndef{\bfZ}{{\mathbf Z}} \ndef{\euZ}{{\EuScript Z}} \ndef{\frZ}{{\mathfrak Z}}

\ndef{\tA}{{\widetilde A}} \ndef{\tcA}{{\widetilde\clA}} \ndef{\ttcA}{\widetilde{\tcA}} \ndef{\sfA}{{\textsf A}} \ndef{\ttA}{\widetilde{\tA}} \ndef{\dzA}{{A^\sharp}}
\ndef{\tB}{{\widetilde B}} \ndef{\tcB}{{\widetilde\clB}} \ndef{\ttcB}{\widetilde{\tcB}} \ndef{\sfB}{{\textsf B}} \ndef{\ttB}{\widetilde{\tB}} \ndef{\dzB}{{B^\sharp}}
\ndef{\tC}{{\widetilde C}} \ndef{\tcC}{{\widetilde\clC}} \ndef{\ttcC}{\widetilde{\tcC}} \ndef{\sfC}{{\textsf C}} \ndef{\ttC}{\widetilde{\tC}} \ndef{\dzC}{{C^\sharp}}
\ndef{\tD}{{\widetilde D}} \ndef{\tcD}{{\widetilde\clD}} \ndef{\ttcD}{\widetilde{\tcD}} \ndef{\sfD}{{\textsf D}} \ndef{\ttD}{\widetilde{\tD}} \ndef{\dzD}{{D^\sharp}}
\ndef{\tE}{{\widetilde E}} \ndef{\tcE}{{\widetilde\clE}} \ndef{\ttcE}{\widetilde{\tcE}} \ndef{\sfE}{{\textsf E}} \ndef{\ttE}{\widetilde{\tE}} \ndef{\dzE}{{E^\sharp}}
\ndef{\tF}{{\widetilde F}} \ndef{\tcF}{{\widetilde\clF}} \ndef{\ttcF}{\widetilde{\tcF}} \ndef{\sfF}{{\textsf F}} \ndef{\ttF}{\widetilde{\tF}} \ndef{\dzF}{{F^\sharp}}
\ndef{\tG}{{\widetilde G}} \ndef{\tcG}{{\widetilde\clG}} \ndef{\ttcG}{\widetilde{\tcG}} \ndef{\sfG}{{\textsf G}} \ndef{\ttG}{\widetilde{\tG}} \ndef{\dzG}{{G^\sharp}}
\ndef{\tH}{{\widetilde H}} \ndef{\tcH}{{\widetilde\clH}} \ndef{\ttcH}{\widetilde{\tcH}} \ndef{\sfH}{{\textsf H}} \ndef{\ttH}{\widetilde{\tH}} \ndef{\dzH}{{H^\sharp}}
\ndef{\tI}{{\widetilde I}} \ndef{\tcI}{{\widetilde\clI}} \ndef{\ttcI}{\widetilde{\tcI}} \ndef{\sfI}{{\textsf I}} \ndef{\ttI}{\widetilde{\tI}} \ndef{\dzI}{{I^\sharp}}
\ndef{\tJ}{{\widetilde J}} \ndef{\tcJ}{{\widetilde\clJ}} \ndef{\ttcJ}{\widetilde{\tcJ}} \ndef{\sfJ}{{\textsf J}} \ndef{\ttJ}{\widetilde{\tJ}} \ndef{\dzJ}{{J^\sharp}}
\ndef{\tK}{{\widetilde K}} \ndef{\tcK}{{\widetilde\clK}} \ndef{\ttcK}{\widetilde{\tcK}} \ndef{\sfK}{{\textsf K}} \ndef{\ttK}{\widetilde{\tK}} \ndef{\dzK}{{K^\sharp}}
\ndef{\tL}{{\widetilde L}} \ndef{\tcL}{{\widetilde\clL}} \ndef{\ttcL}{\widetilde{\tcL}} \ndef{\sfL}{{\textsf L}} \ndef{\ttL}{\widetilde{\tL}} \ndef{\dzL}{{L^\sharp}}
\ndef{\tM}{{\widetilde M}} \ndef{\tcM}{{\widetilde\clM}} \ndef{\ttcM}{\widetilde{\tcM}} \ndef{\sfM}{{\textsf M}} \ndef{\ttM}{\widetilde{\tM}} \ndef{\dzM}{{M^\sharp}}
\ndef{\tN}{{\widetilde N}} \ndef{\tcN}{{\widetilde\clN}} \ndef{\ttcN}{\widetilde{\tcN}} \ndef{\sfN}{{\textsf N}} \ndef{\ttN}{\widetilde{\tN}} \ndef{\dzN}{{N^\sharp}}
\ndef{\tO}{{\widetilde O}} \ndef{\tcO}{{\widetilde\clO}} \ndef{\ttcO}{\widetilde{\tcO}} \ndef{\sfO}{{\textsf O}} \ndef{\ttO}{\widetilde{\tO}} \ndef{\dzO}{{O^\sharp}}
\ndef{\tP}{{\widetilde P}} \ndef{\tcP}{{\widetilde\clP}} \ndef{\ttcP}{\widetilde{\tcP}} \ndef{\sfP}{{\textsf P}} \ndef{\ttP}{\widetilde{\tP}} \ndef{\dzP}{{P^\sharp}}
\ndef{\tQ}{{\widetilde Q}} \ndef{\tcQ}{{\widetilde\clQ}} \ndef{\ttcQ}{\widetilde{\tcQ}} \ndef{\sfQ}{{\textsf Q}} \ndef{\ttQ}{\widetilde{\tQ}} \ndef{\dzQ}{{Q^\sharp}}
\ndef{\tR}{{\widetilde R}} \ndef{\tcR}{{\widetilde\clR}} \ndef{\ttcR}{\widetilde{\tcR}} \ndef{\sfR}{{\textsf R}} \ndef{\ttR}{\widetilde{\tR}} \ndef{\dzR}{{R^\sharp}}
\ndef{\tS}{{\widetilde S}} \ndef{\tcS}{{\widetilde\clS}} \ndef{\ttcS}{\widetilde{\tcS}} \ndef{\sfS}{{\textsf S}} \ndef{\ttS}{\widetilde{\tS}} \ndef{\dzS}{{S^\sharp}}
\ndef{\tT}{{\widetilde T}} \ndef{\tcT}{{\widetilde\clT}} \ndef{\ttcT}{\widetilde{\tcT}} \ndef{\sfT}{{\textsf T}} \ndef{\ttT}{\widetilde{\tT}} \ndef{\dzT}{{T^\sharp}}
\ndef{\tU}{{\widetilde U}} \ndef{\tcU}{{\widetilde\clU}} \ndef{\ttcU}{\widetilde{\tcU}} \ndef{\sfU}{{\textsf U}} \ndef{\ttU}{\widetilde{\tU}} \ndef{\dzU}{{U^\sharp}}
\ndef{\tV}{{\widetilde V}} \ndef{\tcV}{{\widetilde\clV}} \ndef{\ttcV}{\widetilde{\tcV}} \ndef{\sfV}{{\textsf V}} \ndef{\ttV}{\widetilde{\tV}} \ndef{\dzV}{{V^\sharp}}
\ndef{\tW}{{\widetilde W}} \ndef{\tcW}{{\widetilde\clW}} \ndef{\ttcW}{\widetilde{\tcW}} \ndef{\sfW}{{\textsf W}} \ndef{\ttW}{\widetilde{\tW}} \ndef{\dzW}{{W^\sharp}}
\ndef{\tX}{{\widetilde X}} \ndef{\tcX}{{\widetilde\clX}} \ndef{\ttcX}{\widetilde{\tcX}} \ndef{\sfX}{{\textsf X}} \ndef{\ttX}{\widetilde{\tX}} \ndef{\dzX}{{X^\sharp}}
\ndef{\tY}{{\widetilde Y}} \ndef{\tcY}{{\widetilde\clY}} \ndef{\ttcY}{\widetilde{\tcY}} \ndef{\sfY}{{\textsf Y}} \ndef{\ttY}{\widetilde{\tY}} \ndef{\dzY}{{Y^\sharp}}
\ndef{\tZ}{{\widetilde Z}} \ndef{\tcZ}{{\widetilde\clZ}} \ndef{\ttcZ}{\widetilde{\tcZ}} \ndef{\sfZ}{{\textsf Z}} \ndef{\ttZ}{\widetilde{\tZ}} \ndef{\dzZ}{{Z^\sharp}}

\ndef{\bfc}{{\bf c}}

  \ndef{\eps}{\varepsilon}

\let\geq\geqslant
\let\leq\leqslant

\ndef{\lims}[1]{\lim\limits_{#1}}
\ndef{\sums}[1]{\sum\limits_{#1}}
\ndef{\ints}[1]{\int\limits_{#1}}
\ndef{\sups}[1]{\sup\limits_{#1}}
\ndef{\liminfty}[1]{\lims{#1\to\infty}}
\ndef{\suminf}[1]{\sums{#1=1}^\infty}

\ndef{\limo}[1]{\omega\mbox{-}\!\!\!\lims{#1\to\infty}}          
\ndef{\limL}[1]{\rmL\mbox{-}\!\!\!\lims{#1\to\infty}}            
\ndef{\limLOne}[1]{\clL_1\mbox{-}\!\lims{#1}}
\ndef{\tildelimo}[1]{\tilde\omega\mbox{-}\!\!\!\lims{#1\to\infty}}
\ndef{\slim}{\mathrm{s}\mbox{-}\!\!\lim}          
\ndef{\wlim}{\mathrm{w}\mbox{-}\!\lim}          

\ndef{\Aut}{\operatorname{Aut}}      
\ndef{\Ch}{\operatorname{ch}}        
\ndef{\End}{\operatorname{End}}      
\ndef{\Hom}{\operatorname{Hom}}      
\ndef{\Ker}{\operatorname{Ker}}      
\ndef{\Log}{\operatorname{Log}}      
\ndef{\OP}{\operatorname{OP}}        
\ndef{\Op}{\operatorname{Op}}        
\ndef{\Symb}{\operatorname{Symb}}    
\ndef{\Tr}{\operatorname{Tr}}        
\ndef{\Wres}{\operatorname{Wres}}    
\ndef{\cl}{\operatorname{cl}}        
\ndef{\com}{\operatorname{com}}
\ndef{\const}{\operatorname{const}}  
\ndef{\conv}{\operatorname{conv}}    
\rndef{\det}{\operatorname{det}}     

\ndef{\detFK}[1]{\Delta\brs{#1}} 
\ndef{\detFKrel}[2]{\Delta_{#2}\brs{#1}} 

\ndef{\adj}{\operatorname{adj}}    
\ndef{\diag}{\operatorname{diag}}    
\ndef{\dist}{\operatorname{dist}}    
\ndef{\dom}{\operatorname{dom}}      
\ndef{\ec}{\operatorname{ec}}        
\ndef{\id}{1}                        
\ndef{\ind}{\operatorname{ind}}      
\ndef{\mydeg}{\operatorname{deg}}    
\ndef{\op}{\operatorname{op}}
\ndef{\rank}{\operatorname{rank}}
\ndef{\res}{\operatorname{res}}      
\ndef{\rng}{\operatorname{ran}}      
\ndef{\sflow}{\operatorname{sf}}     
\ndef{\isf}{\operatorname{isf}}      
\ndef{\sign}{\operatorname{sign}}    
\ndef{\sgn}{\operatorname{sgn}}      
\ndef{\sing}{\operatorname{sing}}    
\ndef{\supp}{\operatorname{supp}}    
\ndef{\tr}{\operatorname{tr}}        
\ndef{\var}{\operatorname{var}}      
\ndef{\vol}{\operatorname{vol}}      
\ndef{\wn}{\operatorname{wn}}        
\ndef{\wres}{\operatorname{wres}}    
\rndef{\Im}{\operatorname{Im}}       
\rndef{\Re}{\operatorname{Re}}       

\ndef{\prng}[1]{\mathrm R_{#1}} 
\ndef{\pker}[1]{\mathrm N_{#1}} 
\ndef{\rprng}[2]{\mathrm R_{#1}^{#2}}           
\ndef{\rpker}[2]{\mathrm N_{#1}^{#2}}           
\ndef{\rsupp}[1]{\supp_r(#1)}
\ndef{\lsupp}[1]{\supp_l(#1)}
\ndef{\rslv}[1]{R_z(#1)}      
\ndef{\HH}{H}                 
\ndef{\tHH}{\tilde \HH}       
\ndef{\VV}{V}                 
\ndef{\Rz}{R_z}               
\ndef{\tRz}{\tR_z}            
\ndef{\psif}[1]{#1^{[1]}} 
\ndef{\CPlus}[1]{W_{#1}(\mbR)}
\ndef{\bndl}{\xi}                         
\ndef{\bndlA}{\eta}                       
\ndef{\GlueMap}{\varphi}                  
\ndef{\ChartMap}{h}                       

\ndef{\hilb}{\clH}                     
\ndef{\hilba}{\clH^{(a)}}                    
   \ndef{\hilbasargument}{(\hilb)} 
\ndef{\LpH}[1]{\clL^{#1}\hilbasargument}       
\ndef{\saLpH}[1]{\clL_{sa}^{#1}\hilbasargument}       
\ndef{\clBH}{\clB\hilbasargument}              
\ndef{\ubBH}{\clB_1\hilbasargument}            
\ndef{\clCH}{\clC\hilbasargument}              
\ndef{\clKH}{\clK\hilbasargument}              
\ndef{\clFH}{\clF\hilbasargument}              
\ndef{\clUH}{\clU\hilbasargument}              
\ndef{\clCFH}{{\clC\clF}\hilbasargument}       
\ndef{\saBH}{\clB_{sa}\hilbasargument}         
\ndef{\saCH}{\clC_{sa}\hilbasargument}         
\ndef{\saFH}{\clF_{sa}\hilbasargument}         
\ndef{\saKH}{\clK_{sa}\hilbasargument}         
\ndef{\saCFH}{\clC\clF_{sa}\hilbasargument}    
\ndef{\clUFH}{\clU\clF\hilbasargument}         
\ndef{\Uinj}{\clU_{inj}\hilbasargument}        
\ndef{\UFinj}{\clU\clF_{inj}\hilbasargument}   

\ndef{\spproj}[2]{E^{#1}_{#2}}                      
\ndef{\spprojb}[2]{E^{#2}_{#1}}                     

\ndef{\LpN}[1]{\clL^{#1}(\clN,\tau)}     
\ndef{\saLpN}[1]{\clL^{#1}_{sa}(\clN,\tau)} 
\ndef{\rLpN}[1]{L^{#1}(\clN,\tau)}       
\ndef{\clAND}{(\clA,\clN,D)}             
\ndef{\clBA}{{\clB(\clA)}}
\ndef{\saKN}{{\clK_{sa}(\clN,\tau)}}          
\ndef{\clKN}{{\clK(\clN,\tau)}}          
\ndef{\clKtN}{{\clK(\tilde\clN,\tau)}}   
\ndef{\clFN}{{\clF(\clN,\tau)}}          
\ndef{\saFN}{{\clF_{sa}(\clN,\tau)}}     
\ndef{\clPN}{\clP(\clN)}                 
\ndef{\clQN}{\clQ(\clN,\tau)}            
\ndef{\infPN}{{\clP_\tau^\infty(\clN)}}  
\ndef{\clOF}[2]{\clF_{#1\mbox{-}#2}(\clN,\tau)}         
\ndef{\oind}[2]{{\rm \tau\mbox{-}ind}_{#1\mbox{-}#2}}   
\ndef{\tind}{\tau\mbox{-}\ind}                  
\ndef{\DInd}{\ind_{\clD,\tau}}           
\ndef{\BF}{Breuer-Fredholm}              
\ndef{\skewfred}[2]{$(#1\cdot #2)$ $\tau$\tire Fredholm}   
\ndef{\affl}{\eta}                       
\ndef{\vNa}{von Neumann algebra}         
\ndef{\nsf}{faithful normal semifinite } 
\ndef{\taubrs}[1]{\tau\brackets{#1}}     
\ndef{\sqbrs}[1]{[#1]}        
\ndef{\Sqbrs}[1]{\big[#1\big]}        
\ndef{\SqBrs}[1]{\Big[#1\Big]}        

\ndef{\domd}{\bigcap\limits_{n\ge 0} \dom\;\delta^n}         
\ndef{\DiffOP}{{\rm \clD}}
\ndef{\ADA}{\clA \cup [\clD,\clA]}
\ndef{\DixIdeal}[1]{\LpH{#1,\infty}}               
\ndef{\dixideal}{\ell^{1,\infty}}                  
\ndef{\WDixIdeal}{\LpH{1,\mathrm w}}               
\ndef{\DixIdealPos}[1]{\DixIdeal{#1}_+}            
\ndef{\DixIdealN}[1]{\LpN{#1,\infty}}              
\ndef{\DixIdealNPar}[2]{\clL^{#1,\infty}_{#2}(\clN,\tau)}    
\ndef{\DixIdealNPos}[1]{\LpN{#1,\infty}_+}                   
\ndef{\TrD}{\Tr_\omega}                                      
\ndef{\tauD}{{\tau_\omega}}                                  
\ndef{\ILog}{\frac 1{\log(1+t)}}
\ndef{\ILogN}{\frac 1{\log(1+N)}}
\ndef{\DixNorm}[1]{\norm{#1}_{(1,\infty)}}                   
\ndef{\DixInt}[1]{\ints 0^t \mu_s(#1)\,ds}
\ndef{\DixIntL}[1]{\ints 0^{\lambda_{1/t}(#1)}\mu_s(#1)\,ds}
    \ndef{\SmallIdeal}{{\clL^{1, \mathrm w}}}
    \ndef{\SmallIdealMeas}{{\clL^{1, \mathrm w}_m}}
    \ndef{\DixIntII}[1]{\ints 0^t \mu_s(#1)\,ds}
    \ndef{\DixIntf}[1]{\Phi_t(#1)}
    \ndef{\DixIntg}[1]{\Psi_t(#1)}

\ndef{\lpi}{\clL^{1,\pi}(\clN,\tau)}
\ndef{\IIinfty}{$\mathrm{II}_\infty$\ }

\ndef{\fourier}[1]{\clF(#1)}          
\ndef{\HaarMeasBohrs}{\nu}            
\ndef{\BrownsMeas}{\mu}               
\ndef{\BohrCont}[1]{\tilde{#1}}       
\ndef{\APMean}{{M}}                   
\ndef{\CDSS}{{\clA_B}}                
\ndef{\matr}{{\rm Mat}}               
\ndef{\seque}[1]{\ensuremath{\{#1_j\}_{j=1}^\infty}}    
\ndef{\sequen}[2]{\ensuremath{\{#1_#2\}_{#2=1}^\infty}}    
\ndef{\Seque}[1]{\ensuremath{\left(#1_0,#1_1,#1_2,\dots\right)}}    
\ndef{\Cesaro}{H}                           
\ndef{\CesaroRPlus}{M}                      
\ndef{\Dilation}{D}                         
\ndef{\Shift}{T}                            

\ndef{\norm}[1]{\left\Vert#1\right\Vert}    
\ndef{\TrNorm}[1]{\norm{#1}_1}              
\ndef{\HSNorm}[1]{\norm{#1}_2}              
\ndef{\InftyNorm}[1]{\norm{#1}_\infty}      
\ndef{\normQN}[1]{\norm{#1}_{\clQN}}        
\ndef{\clLnorm}[1]{\norm{#1}_{1,\infty}}    

\ndef{\ccurve}{\gamma}                      

\ndef{\abs}[1]{\left\lvert#1\right\rvert}   
\ndef{\set}[1]{\left\{#1\right\}}           
\ndef{\brackets}[1]{\left(#1\right)}        
\ndef{\brs}[1]{\brackets{#1}}               
\ndef{\Brs}[1]{\big(#1\big)}                
\ndef{\BRS}[1]{\Big(#1\Big)}                
\ndef{\scal}[2]{\left\la #1,#2\right\ra}               
\ndef{\precprec}{\prec\!\!\!\prec}
\ndef{\qeq}{\stackrel?=}
\ndef{\spectrum}[1]{\sigma_{#1}} 
\ndef{\spectruma}[1]{\sigma^{(a)}_{#1}} 
\ndef{\numrange}[1]{\mathrm{W}(#1)}                         
\rndef{\emptyset}{\varnothing}                              
\ndef{\csupp}{c}                           
\ndef{\closure}[1]{\overline{#1}}
\ndef{\linspan}[1]{\mathrm{span}\ {#1}}
\ndef{\bddborel}[1]{B(#1)}                 
\ndef{\charfunc}{\chi}
\rndef{\ln}{\log}
\ndef{\FrDer}{\euD}                        
\ndef{\LieDer}[1]{\pounds_{#1}\,}          
\ndef{\dds}{\left.\frac d{ds} \right|_{s = 0}}
\ndef{\ortcmp}[1]{#1^{\scriptscriptstyle \perp}}            
\ndef{\Laplace}{\Delta}                    

\ndef{\matrPQ}[3]
{
    \left(
      \begin{array}{cc}
        #1_{11} & #1_{12} \\
        #1_{21} & #1_{22}
      \end{array}
    \right)_{[#2,#3]}
}

\newcounter{margcomcount}
\ndef{\margcom}[1]{\marginpar{\bf \small #1} \addtocounter{margcomcount}{1}
   \index{******** \indexcom{{COMMENT}}}}

\newcounter{margproof}
\ndef{\margproof}{\marginpar{\bf \small Proof} \addtocounter{margproof}{1}
  \index{******** \indexcom{{\bf PROOF}}}}

\newcounter{margdetails}
\ndef{\margdetails}{\marginpar{\bf details} \addtocounter{margdetails}{1}
  \index{******** \indexcom{{\bf DETAILS}}}}

\newcounter{margproofb}
\ndef{\margproofb}{\marginpar{\bf \small Proof (B)} \addtocounter{margproofb}{1}
  \index{******** \indexcom{{PROOF (B)}}}}

\newcounter{margdetailsb}
\ndef{\margdetailsb}{\marginpar{\bf \small Details (B)} \addtocounter{margdetailsb}{1}
  \index{******** \indexcom{{DETAILS (B)}}}}

\ndef{\mytimes}{\!\times\!}
\ndef{\sss}[1]{\subsubsection{}\label{#1}}
\rndef{\phi}{\varphi}
\ndef{\OpenUnitDisk}{D}
\ndef{\RHS}{RHS}                            
\ndef{\LHS}{LHS} 
\ndef{\ttt}{\Leftrightarrow}
\ndef{\then}{\Rightarrow}
\ndef{\tto}{\longrightarrow}
\ndef{\nno}{\nonumber\\}
\ndef{\newn}[1]{\index{#1} \emph{#1}}       
\ndef{\la}{\langle}
\ndef{\ra}{\rangle}
\ndef{\dbar}{{\;\bar{\phantom{o}} \!\!\!\! d}}
\ndef{\stl}[1]{\stackrel{\vbox to 0pt{\vss\hbox{$\scriptstyle #1$}}}}
\ndef{\mathcomment}[1]{{\scriptstyle\text{(#1)}}\qquad}        
\ndef{\details}[1]{\smallskip\begin{center} {\bf Here:} #1\end{center}\medskip}
\ndef{\indexcom}[1]{ --- #1}
\ndef{\longsim}{\ \sim \ }              
\ndef{\tire}{-}              
\ndef{\intinfinf}{\int_{-\infty}^\infty}
\ndef{\refnsftrace}{\cite[V.\,2.\,1]{TakI}} 
\ndef{\refaffloper}{\cite[IV.\,5, Exercise 3]{TakI}} 
\ndef{\refsemifinvNa}{\cite[V.\,1.\,21]{TakI}} 
\ndef{\reftaumeasurable}{\cite[Definition 1.2]{FK86PJM}} 
\ndef{\reftautraceclassaffl}{\cite[V.2, p.\,320]{TakI}} 
\ndef{\refinvoperideal}{\cite[Appendix A.2]{CP2}} 
\ndef{\reftautracenorm}{\cite[V.2, p.\,320]{TakI}} 
\ndef{\reftaucompact}{\cite{}} 
\ndef{\reftauFredholm}{\cite[Appendix B]{PR94JFA}} 

     \ndef{\npartial}{\slash\!\!\!\partial}
     \ndef{\Heis}{\operatorname{Heis}}
     \ndef{\Solv}{\operatorname{Solv}}
     \ndef{\Spin}{\operatorname{Spin}}
     \ndef{\SO}{\operatorname{SO}}
     \ndef{\Index}{\operatorname{index}}

             \ndef{\coker}{{\mbox coker}}
             \ndef{\p}{\partial}
             \ndef{\dd}{|\clD|}
             \ndef{\n}{\parallel}

     \ndef{\gf}[2]{\genfrac{}{}{0pt}{}{#1}{#2}}
     \ndef{\ta}{\widetilde{\alpha}}
     \ndef{\tb}{\widetilde{\beta}}
     \ndef{\txi}{\widetilde{\xi}}
     \ndef{\tk}{\widetilde{K}}
     \ndef{\CGh}{\widetilde{\CG}}
     \ndef{\boe}{{\bf e}}\ndef{\bt}{{\bf t}}
     \ndef{\vth}{\vartheta}
     \ndef{\db}{\overline{\partial}}
     \ndef{\hV}{\hat{V}}
     \ndef{\cag}{{\clA^\Gamma}}
     \ndef{\sind}{\sigma{\rm -ind}}

\let\LatexCite=\cite  

\let\ifnumref\iftrue 

\ndef{\ifuncited}[4]{\expandafter\ifx\csname used#4\endcsname\relax}

\ndef{\ifcited}[4]{\expandafter\ifx\csname used#4\endcsname\relax\else}

  \ndef{\papertitle}[1]{ \emph{#1}, }
  \ndef{\paperauthor}[2]{#2}  
  \ndef{\pbbi}[9]{%
      \ifcited{#1}{#2}{#3}{#5}%
        \ifnumref%
          \bibitem{#5}\paperauthor{#1}{#6},\papertitle{#7}#8.%
        \else%
          \advance #9 by 1%
          \ifnum#9<1%
            \bibitem[#4]{#5}\paperauthor{#1}{#6}, \papertitle{#7}#8.%
          \else%
            \bibitem[#4$\!_{\the#9}\!$]{#5}\paperauthor{#1}{#6},\papertitle{#7}#8.%
          \fi%
        \fi%
      \fi%
  }
  \ndef{\mbbi}[8]{%
     \ifcited{#1}{#2}{#3}{#5}%
        \ifnumref%
          \bibitem{#5}\paperauthor{#1}{#6},\papertitle{#7}#8.%
        \else%
          \bibitem[#4]{#5}\paperauthor{#1}{#6},\papertitle{#7}#8.%
        \fi%
     \fi%
  }

\ndef{\AddCite}[1]{%
   \ifuncited{0}{0}{0}{#1}%
     \expandafter\gdef\csname used#1\endcsname {}%
   \fi%
}

\def\ProcessCite#1,{%
     \ifx\relax#1%
         \let\next=\relax%
     \else%
         \AddCite{#1}%
         \let\next=\ProcessCite%
     \fi%
     \next%
}

\ndef{\AddCites}[1]{\ProcessCite#1,\relax,}

\ndef{\CiteWithoutExtension}[1]{%
   \AddCites{#1}%
   \LatexCite{#1}%
}

\def\CiteWithExtension[#1]#2{%
   \AddCites{#2}%
   \LatexCite[#1]{#2}%
}

\ndef{\CleverCite}{%
    \ifx\NChar[ %
       \let\MyCite=\CiteWithExtension %
    \else %
       \let\MyCite=\CiteWithoutExtension %
    \fi %
    \MyCite%
}

\renewcommand{\cite}{\futurelet\NChar\CleverCite}

      \ndef{\volume}[1]{{\bf #1}}
      \ndef{\VolYearPP}[3]{\ifnum#2=0 (to appear)\else\volume{#1} (#2), #3\fi}
      \ndef{\VolNoYearPP}[4]{\ifnum#3=0 (to appear)\else\volume{#1} #2 (#3), #4\fi}
      \ndef{\libcode}[1]{}

\ndef{\jnActaMath}[3]{Acta Math. \VolYearPP{#1}{#2}{#3}}                       
\ndef{\jnAdvMath}[3]{Adv. in~Math. \VolYearPP{#1}{#2}{#3}}                     
\ndef{\jnAlgAnal}[3]{Algebra i~Analiz \VolYearPP{#1}{#2}{#3}}
\ndef{\jnAmerMathMonth}[3]{Amer. Math. Monthly \VolYearPP{#1}{#2}{#3}}         
\ndef{\jnAnnMath}[4]{Ann. of~Math. \VolNoYearPP{#1}{#2}{#3}{#4}}               
\ndef{\jnAnalMath}[3]{J. Anal. Math. \VolYearPP{#1}{#2}{#3}}                   
\ndef{\jnBullLondMathSoc}[3]{Bull. London Math. Soc. \VolYearPP{#1}{#2}{#3}}   
\ndef{\jnBullAMS}[3]{Bull. Amer. Math. Soc. \VolYearPP{#1}{#2}{#3}}   
\ndef{\jnCanMathBull}[3]{Canad. Math. Bull. \VolYearPP{#1}{#2}{#3}}            
\ndef{\jnCanMath}[3]{Canad. J.~Math. \VolYearPP{#1}{#2}{#3}}             
\ndef{\jnCommMathPhys}[3]{Comm. Math. Phys \VolYearPP{#1}{#2}{#3}}             
\ndef{\jnCommPDE}[3]{Comm. Partial Differential Equations \VolYearPP{#1}{#2}{#3}}             
\ndef{\jnComptRendue}[3]{C.\,R.~Acad. Sci. Paris S\'er. A-B \VolYearPP{#1}{#2}{#3}}      
\ndef{\jnContMath}[3]{Contemporary Math. \VolYearPP{#1}{#2}{#3}}               %
\ndef{\jnDukeMJ}[3]{Duke Math. J. \VolYearPP{#1}{#2}{#3}}
\ndef{\jnDiffGeom}[3]{J.~Diff. Geom. \VolYearPP{#1}{#2}{#3}}                   
\ndef{\jnErgodicTheory}[3]{Ergodic Theory and Dynamical Systems \VolYearPP{#1}{#2}{#3}} 
\ndef{\jnFuncAnal}[3]{J.~Functional Analysis \VolYearPP{#1}{#2}{#3}}           
\ndef{\jnFunkAnalPril}[4]{Funct. Anal. Appl. \VolNoYearPP{#1}{#2}{#3}{#4}}  
\ndef{\jnGAFA}[3]{GAFA \VolYearPP{#1}{#2}{#3}}                                 
\ndef{\jnIHES}[3]{IHES Publ. Math. (Paris) \VolYearPP{#1}{#2}{#3}}             
\ndef{\jnIEOT}[3]{Integral Equations Operator Theory   \VolYearPP{#1}{#2}{#3}} 
\ndef{\jnIsrMath}[3]{Israel J.~Math. \VolYearPP{#1}{#2}{#3}}                   
\ndef{\jnKTheory}[3]{K-Theory \VolYearPP{#1}{#2}{#3}}                          
\ndef{\jnLetMathPhys}[3]{Lett. Math. Phys. \VolYearPP{#1}{#2}{#3}}             
\ndef{\jnMathAnn}[3]{Math. Ann. \VolYearPP{#1}{#2}{#3}}                        
\ndef{\jnMathAnalAppl}[3]{J.~Math. Anal. and Appl. \VolYearPP{#1}{#2}{#3}}     
\ndef{\jnMathNachr}[3]{Math. Nachr. \VolYearPP{#1}{#2}{#3}}
\ndef{\jnMathPhys}[3]{J. Math. Phys. \VolYearPP{#1}{#2}{#3}}
\ndef{\jnMathSocJap}[3]{J. Math. Soc. Japan \VolYearPP{#1}{#2}{#3}}
\ndef{\jnOperTheory}[3]{J.~Operator Theory \VolYearPP{#1}{#2}{#3}}             
\ndef{\jnPacJMath}[3]{Pacific J.~Math. \VolYearPP{#1}{#2}{#3}}                  
\ndef{\jnPositivity}[3]{Positivity \VolYearPP{#1}{#2}{#3}}
\ndef{\jnProcAmerMS}[3]{Proc. Amer. Math. Soc. \VolYearPP{#1}{#2}{#3}}         
\ndef{\jnProcCambPhilSoc}[3]{Math. Proc. Camb. Phil. Soc. \VolYearPP{#1}{#2}{#3}}
\ndef{\jnReineAngew}[3]{J.~Reine Angew. Math. \VolYearPP{#1}{#2}{#3}}          
\ndef{\jnTokyoMath}[3]{Tokyo J.~Math. \VolYearPP{#1}{#2}{#3}}
\ndef{\jnTopology}[3]{Topology \VolYearPP{#1}{#2}{#3}}
\ndef{\jnTransAmerMathSoc}[3]{Trans. Amer. Math. Soc. \VolYearPP{#1}{#2}{#3}}
\ndef{\jnIzvANSSSR}[3]{Izv. Akad. Nauk SSSR, Ser. Mat. \VolYearPP{#1}{#2}{#3}}
\ndef{\jnIzvVyshUchZav}[3]{Izv. Vyssh. Uch. Zav., Mat. \VolYearPP{#1}{#2}{#3} (Russian)}
\ndef{\jnIzdatLenUniv}[2]{Izdat. Leningrad. Univ., Leningrad, (#1), #2 (Russian)}
\ndef{\jnFieldsInsComm}[3]{Fields Inst. Comm. \VolYearPP{#1}{#2}{#3}}
\ndef{\jnDoklANSSSR}[3]{Dokl. Akad. Nauk SSSR \VolYearPP{#1}{#2}{#3}}
\ndef{\jnMatZametki}[3]{Matem. zametki \VolYearPP{#1}{#2}{#3}}
\ndef{\jnRussMathSurvey}[3]{Russian Math. Surveys \VolYearPP{#1}{#2}{#3}}
\ndef{\jnSibMathJ}[3]{Sib. Math.~J. \VolYearPP{#1}{#2}{#3}}
\ndef{\jnSovMath}[3]{J.~Soviet math. \VolYearPP{#1}{#2}{#3}}
\ndef{\jnTransMoscMathSoc}[3]{Trans. Moscow Math. Soc. \VolYearPP{#1}{#2}{#3}}
\ndef{\jnUMN}[3]{Uspekhi Mat. Nauk \VolYearPP{#1}{#2}{#3}}

\ndef{\bkTransMathMon}[2]{Trans. Math. Monographs, AMS, \volume{#1}, #2}

\ndef{\pbBirkhauser}[1]{Birkh\"auser, Boston, #1}
\ndef{\pbFactorial}[1]{Moscow, Factorial, #1}
\ndef{\pbGauthier}[1]{Gauthier-Villars, Paris, #1}
\ndef{\pbNauka}[1]{Moscow, Nauka, #1 (Russian)}
\ndef{\pbNaukaR}[1]{Москва, Наука, #1}
\ndef{\pbPrinceton}[1]{Princeton University Press, Princeton, New Jersey, #1}
\ndef{\pbPublPerish}[1]{Publish or Perish Inc., Berkeley, #1}
\ndef{\pbSpringer}[1]{Springer-Verlag, #1}

\ndef{\myauthor}[1]{\mbox{#1}}

\ndef{\Agmon}{\myauthor{Sh.\,Agmon}}
\ndef{\Ahiezer}{\myauthor{N.\,I.\,Ahiezer}}
\ndef{\Arazy}{\myauthor{J.\,Arazy}}
\ndef{\Astashkin}{\myauthor{S.\,V.\,Astashkin}}
\ndef{\Atiyah}{\myauthor{M.\,Atiyah}}
\ndef{\Avron}{\myauthor{J.\,E.\,Avron}}
\ndef{\Azamov}{\myauthor{N.\,A.\,Azamov}}
\ndef{\Banach}{\myauthor{S.\,Banach}}
\ndef{\Benameur}{\myauthor{M-T.\,Benameur}}
\ndef{\Bennett}{\myauthor{C.\,Bennett}}
\ndef{\Berezin}{\myauthor{F.\,A.\,Berezin}}
\ndef{\Berline}{\myauthor{N.\,Berline}}
\ndef{\Birman}{\myauthor{M.\,Sh.\,Birman}}
\ndef{\Blackadar}{\myauthor{B.\,Blackadar}}
\ndef{\Bogolyubov}{\myauthor{N.\,N.\,Bogolyubov}}
\ndef{\Bonsall}{\myauthor{F.\,F.\,Bonsall}}
\ndef{\Bony}{\myauthor{J.\,F.\,Bony}}
\ndef{\BoosBavnbek}{\myauthor{B.\,Boo$\beta$-Bavnbek}}
\ndef{\Bott}{\myauthor{R.\,Bott}}
\ndef{\Bratteli}{\myauthor{O.\,Bratteli}}
\ndef{\Bredon}{\myauthor{G.\,E.\,Bredon}}
\ndef{\Breuer}{\myauthor{M.\,Breuer}}
\ndef{\Brown}{\myauthor{L.\,G.\,Brown}}
\ndef{\Bruneau}{\myauthor{V.\,Bruneau}}
\ndef{\Buslaev}{\myauthor{V.\,S.\,Buslaev}}
\ndef{\Carey}{\myauthor{A.\,L.\,Carey}}
\ndef{\CareyRW}{\myauthor{R.\,W.\,Carey}} 
\ndef{\Cartan}{\myauthor{H.\,Cartan}}
\ndef{\Chilin}{\myauthor{V.\,I.\,Chilin}}
\ndef{\Coburn}{\myauthor{L.\,A.\,Coburn}}
\ndef{\Connes}{\myauthor{A.\,Connes}}
\ndef{\Cornfeld}{\myauthor{I.\,P.\,Cornfeld}}
\ndef{\Daletskii}{\myauthor{Yu.\,L.\,Daletski\u\i}}   
\ndef{\Dixmier}{\myauthor{J.\,Dixmier}}
\ndef{\DoddsPG}{\myauthor{P.\,G.\,Dodds}}
\ndef{\DoddsTK}{\myauthor{T.\,K.\,Dodds}}
\ndef{\Douglas}{\myauthor{R.\,G.\,Douglas}}
\ndef{\Dubrovin}{\myauthor{B.\,A.\,Dubrovin}}
\ndef{\Dugundji}{\myauthor{J.\,Dugundji}}
\ndef{\Duncan}{\myauthor{J.\,Duncan}}
\ndef{\Dunford}{\myauthor{N.\,Dunford}}
\ndef{\Dykema}{\myauthor{K.\,J.\,Dykema}}
\ndef{\Edwards}{\myauthor{R.\,E.\,Edwards}}
\ndef{\Eilenberg}{\myauthor{S.\,Eilenberg}}
\ndef{\Entina}{\myauthor{S.\,B.\,\`Entina}}
\ndef{\Fack}{\myauthor{T.\,Fack}} 
\ndef{\Faddeev}{\myauthor{L.\,D.\,Faddeev}}
\ndef{\Farber}{\myauthor{M.\,Farber}}
\ndef{\Farforovskaya}{\myauthor{Yu.\,B.\,Farforovskaya}}
\ndef{\Federer}{\myauthor{H.\,Federer}}
\ndef{\Fedosov}{\myauthor{B.\,V.\,Fedosov}}
\ndef{\Figiel}{\myauthor{T.\,Figiel}} 
\ndef{\Figueroa}{\myauthor{H.\,Figueroa}}
\ndef{\Fillmore}{\myauthor{P.\,A.\,Fillmore}}
\ndef{\Fomenko}{\myauthor{A.\,T.\,Fomenko}} 
\ndef{\Fomin}{\myauthor{S.\,V.\,Fomin}}
\ndef{\Frohlich}{\myauthor{J.\,Fr\"ohlich}}
\ndef{\Fuglede}{\myauthor{B.\,Fuglede}}
\ndef{\Furutani}{\myauthor{K.\,Furutani}}
\ndef{\Gelfand}{\myauthor{I.\,M.\,Gelfand}}
\ndef{\Gesztesy}{\myauthor{F.\,Gesztesy}}     
\ndef{\Getzler}{\myauthor{E.\,Getzler}} 
\ndef{\Gilkey}{\myauthor{P.\,B.\,Gilkey}}
\ndef{\Gitler}{\myauthor{S.\,Gitler}}
\ndef{\Glazman}{\myauthor{I.\,M.\,Glazman}}
\ndef{\Glimm}{\myauthor{J.\,Glimm}}
\ndef{\Gohberg}{\myauthor{I.\,C.\,Gohberg}}
\ndef{\Goldshtein}{\myauthor{Ya.\,Goldshtein}}
\ndef{\Golze}{\myauthor{F.\,Golze}}
\ndef{\GraciaBondia}{\myauthor{J.\,M.\,Gracia-Bond\'{i}a}}
\ndef{\Greenleaf}{\myauthor{F.\,P.\,Greenleaf}}
\ndef{\Gromov}{\myauthor{M.\,Gromov}}
\ndef{\Gunning}{\myauthor{R.\,C.\,Gunning}}
\ndef{\Haagerup}{\myauthor{U.\,Haagerup}}
\ndef{\Haag}{\myauthor{R.\,Haag}}
\ndef{\Halmos}{\myauthor{Halmos}}
\ndef{\Hardy}{\myauthor{G.\,H.\,Hardy}}
\ndef{\Herbst}{\myauthor{I.\,W.\,Herbst}}
\ndef{\Higson}{\myauthor{N.\,Higson}}  
\ndef{\Hoermander}{\myauthor{L.\,Hoermander}} 
\ndef{\Hoffman}{\myauthor{K.\,Hoffman}} 
\ndef{\Ito}{\myauthor{K.\,Ito}}
\ndef{\Jaffe}{\myauthor{A.\,Jaffe}}
\ndef{\James}{\myauthor{I.\,M.\,James}}
\ndef{\Javrjan}{\myauthor{V.\,A.\,Javrjan}}
\ndef{\Kadison}{\myauthor{R.\,V.\,Kadison}}
\ndef{\Kalton}{\myauthor{N.\,J.\,Kalton}} 
\ndef{\Kato}{\myauthor{T.\,Kato}} 
\ndef{\Kobayashi}{\myauthor{S.\,Kobayashi}}
\ndef{\Koplienko}{\myauthor{L.\,S.\,Koplienko}}
\ndef{\Korotyaev}{\myauthor{E.\,Korotyaev}}
\ndef{\Kosaki}{\myauthor{H.\,Kosaki}}
\ndef{\Kostrykin}{\myauthor{Kostrykin}}
\ndef{\Kotani}{\myauthor{S.\,Kotani}}
\ndef{\Krein}{\myauthor{Kre\u\i n}}
\ndef{\KreinMG}{\myauthor{M.\,G.\,Kre\u\i n}}
\ndef{\KreinSG}{\myauthor{S.\,G.\,Kre\u\i n}}
\ndef{\Kuroda}{\myauthor{S.\,T.\,Kuroda}}
\ndef{\Leichtnam}{\myauthor{E.\,Leichtnam}}
\ndef{\Lesch}{\myauthor{M.\,Lesch}}
\ndef{\Lesniewski}{\myauthor{A.\,Lesniewski}}
\ndef{\Levitan}{\myauthor{B.\,M.\,Levitan}}
\ndef{\Lidskii}{\myauthor{V.\,B.\,Lidskii}}
\ndef{\Lifshits}{\myauthor{I.\,M.\,Lifshits}}
\ndef{\Lindenstrauss}{\myauthor{J.\,Lindenstrauss}}
\ndef{\Loday}{\myauthor{J.-L.\,Loday}}
\ndef{\Lord}{\myauthor{S.\,Lord}}      
\ndef{\Lorentz}{\myauthor{G.\,Lorentz}}
\ndef{\Magnus}{\myauthor{W.\,Magnus}}
\ndef{\Makarov}{\myauthor{K.\,A.\,Makarov}}
\ndef{\Mathai}{\myauthor{V.\,Mathai}}         
\ndef{\McKean}{\myauthor{H.\,P.\,McKean}}
\ndef{\Mishchenko}{\myauthor{A.\,S.\,Mishchenko}}
\ndef{\Molchanov}{\myauthor{S.\,A.\,Molchanov}}
\ndef{\Moore}{\myauthor{C.\,C.\,Moore}}
\ndef{\Moscovici}{\myauthor{H.\,Moscovici}}  
\ndef{\Motovilov}{\myauthor{A.\,K.\,Motovilov}}
\ndef{\Moyer}{\myauthor{R.\,D.\,Moyer}}
\ndef{\Naboko}{\myauthor{S.\,N.\,Naboko}}
\ndef{\Narasimhan}{\myauthor{R.\,Narasimhan}}
\ndef{\Nomizu}{\myauthor{K.\,Nomizu}}
\ndef{\Novikov}{\myauthor{S.\,P.\,Novikov}}
\ndef{\Osterwalder}{\myauthor{K.\,Osterwalder}}
\ndef{\Patodi}{\myauthor{V.\,Patodi}}
\ndef{\Pagter}{\myauthor{B.\,de~Pagter}}  
\ndef{\Pastur}{\myauthor{L.\,A.\,Pastur}}  
\ndef{\Pavlov}{\myauthor{B.\,S.\,Pavlov}}
\ndef{\Pedersen}{\myauthor{G.\,K.\,Pedersen}}
\ndef{\Peller}{\myauthor{V.\,V.\,Peller}}
\ndef{\Perera}{\myauthor{V.\,S.\,Perera}}
\ndef{\Petunin}{\myauthor{Ju.\,I.\,Petunin}}
\ndef{\Phillips}{\myauthor{J.\,Phillips}}  
\ndef{\Piazza}{\myauthor{P.\,Piazza}}   
\ndef{\Pincus}{\myauthor{J.\,D.\,Pincus}}   
\ndef{\Poincare}{Poincar\'e}
\ndef{\Postnikov}{\myauthor{M.\,M.\,Postnikov}} 
\ndef{\Prinzis}{\myauthor{R.\,Prinzis}}
\ndef{\Privalov}{\myauthor{I.\,I.\,Privalov}}
\ndef{\Pushnitski}{\myauthor{A.\,B.\,Pushnitski}} 
\ndef{\Raeburn}{\myauthor{I.\,Raeburn}}
\ndef{\Raikov}{\myauthor{G.\,Raikov}}
\ndef{\Reed}{\myauthor{M.\,Reed}}
\ndef{\Rennie}{\myauthor{A.\,Rennie}}
\ndef{\Rickart}{\myauthor{C.\,E.\,Rickart}}
\ndef{\Riesz}{\myauthor{F.\,Riesz}}
\ndef{\Ringrose}{\myauthor{J.\,Ringrose}}
\ndef{\Robinson}{\myauthor{D.\,Robinson}}
\ndef{\Rossi}{\myauthor{H.\,Rossi}}
\ndef{\Rudin}{\myauthor{W.\,Rudin}}
\ndef{\Ruelle}{\myauthor{D.\,Ruelle}}
\ndef{\Ruzhansky}{\myauthor{M.\,Ruzhansky}}
\ndef{\Sakai}{\myauthor{Sh.\,Sakai}}
\ndef{\Sargsjan}{\myauthor{I.\,S.\,Sargsjan}}
\ndef{\Sato}{\myauthor{H.\,Sato}}
\ndef{\Schaeffer}{\myauthor{D.\,G.\,Schaeffer}}
\ndef{\Schluchtermann}{\myauthor{G.\,Schluchtermann}}
\ndef{\Schochet}{\myauthor{C.\,Schochet}}
\ndef{\Schrodinger}{\myauthor{E.\,Schr\"odinger}}
\ndef{\Schroodinger}{\myauthor{Schr\"odinger}}
\ndef{\Schrohe}{\myauthor{E.\,Schrohe}}
\ndef{\Schwartz}{\myauthor{J.\,T.\,Schwartz}}
\ndef{\Sedaev}{\myauthor{A.\,A.\,Sedaev}}
\ndef{\Seiler}{\myauthor{R.\,Seiler}}
\ndef{\Semenov}{\myauthor{E.\,M.\,Semenov}}
\ndef{\Shabat}{\myauthor{B.\,V.\,Shabat}}
\ndef{\Shafarevich}{\myauthor{I.\,R.\,Shafarevich}}
\ndef{\Sharpley}{\myauthor{R.\,Sharpley}}
\ndef{\Shilov}{\myauthor{G.\,E.\,Shilov}}
\ndef{\Shirkov}{\myauthor{D.\,V.\,Shirkov}}
\ndef{\Shubin}{\myauthor{M.\,A.\,Shubin}}
\ndef{\Silverman}{\myauthor{H.\,Silverman}}
\ndef{\Simon}{\myauthor{B.\,Simon}}
\ndef{\Sinai}{\myauthor{Ya.\,G.\,Sinai}}
\ndef{\Singer}{\myauthor{I.\,M.\,Singer}}
\ndef{\Solomyak}{\myauthor{M.\,Z.\,Solomyak}}
\ndef{\Soloviev}{\myauthor{Yu.\,P.\,Soloviev}}
\ndef{\Spivak}{\myauthor{M.\,Spivak}}
\ndef{\Stenkin}{\myauthor{V.\,V.\,Sten'kin}}
\ndef{\Stratila}{\myauthor{S.\,Stratila}}
\ndef{\Sucheston}{\myauthor{L.\,Sucheston}}
\ndef{\Sukochev}{\myauthor{F.\,A.\,Sukochev}}
\ndef{\Switzer}{\myauthor{R.\,M.\,Switzer}}
\ndef{\SzNagy}{\myauthor{B.\,Sz.-Nagy}}
\ndef{\Takesaki}{\myauthor{M.\,Takesaki}}
\ndef{\Taylor}{\myauthor{M.\,E.\,Taylor}}
\ndef{\Treves}{\myauthor{F.\,Treves}}
\ndef{\Troitsky}{\myauthor{E.\,V.\,Troitsky}}
\ndef{\Tzafriri}{\myauthor{L.\,Tzafriri}}
\ndef{\Varilly}{\myauthor{J.\,C.\,V\'{a}rilly}}
\ndef{\Vergne}{\myauthor{M.\,Vergne}}
\ndef{\Vladimirov}{\myauthor{V.\,S.\,Vladimirov}}
\ndef{\Voiculescu}{\myauthor{D.\,Voiculescu}}
\ndef{\Weiss}{\myauthor{G.\,Weiss}}
\ndef{\Wells}{\myauthor{R.\,O.\,Wells}}
\ndef{\Williams}{\myauthor{J.\,P.\,Williams}}
\ndef{\Winkler}{\myauthor{S.\,Winkler}}
\ndef{\Witten}{\myauthor{E.\,Witten}}
\ndef{\Wodzicki}{\myauthor{M.\,Wodzicki}}
\ndef{\Wojciechowski}{\myauthor{K.\,P.\,Wojciechowski}}
\ndef{\Yafaev}{\myauthor{D.\,R.\,Yafaev}}
\ndef{\Yosida}{\myauthor{K.\,Yosida}}
\ndef{\Zsido}{\myauthor{L.\,Zsido}}

\newcommand{\Texp}{\mathrm{T}\!\exp}

\newcommand{\xia}{\xi^{(a)}}
\newcommand{\xis}{\xi^{(s)}}
\newcommand{\mua}{\mu^{(a)}}
\newcommand{\mus}{\mu^{(s)}}
\newcommand{\LambdaA}{{\Lambda_\clA}}

\newcommand{\dm}{\nu} 

\rndef{\LpH}[1]{\frS_{#1}\hilbasargument}       

\begin{document}
\title[Pushnitski's $\mu$\tire invariant]{Pushnitski's $\mu$\tire invariant and Schr\"odinger operators with embedded eigenvalues}
\author{\Azamov}
\address{School of Informatics and Engineering
   \\ Flinders University of South Australia
   \\ Bedford Park, 5042, SA Australia.}
\email{azam0001@infoeng.flinders.edu.au}
\keywords{Spectral shift function, scattering matrix, infinitesimal spectral flow,
    trace compatible operators, Birman-\Krein\ formula
}
\subjclass[2000]{ 
    Primary 47A55; 
    Secondary 47A11 
}
\begin{abstract} In this note, under a certain assumption on an affine space of operators,
which admit embedded eigenvalues, it is shown that
the singular part of the spectral shift function of any pair of operators from this space
is an integer-valued function.
The proof uses a natural decomposition of Pushnitski's $\mu$-invariant into "absolutely continuous"
and "singular" parts. As a corollary, the Birman-Krein formula follows.
\end{abstract}
\maketitle
\begin{center} {\small \sf\today} \end{center}

\section*{Introduction}
This work is a continuation of my previous note \cite{Az}.
For notation used here one should consult \cite{Az}.

Pushnitski's $\mu$\tire invariant $\mu(\theta,\lambda)$ \cite{Pu01FA} can be defined as spectral flow
of eigenvalues of the scattering matrix $S(\lambda+iy;H_1,H_0)$
(defined appropriately for complex values $\lambda+iy$) through the point $e^{i\theta},$
as $y$ goes from $0+$ to $\infty.$
There is a connection between SSF and Pushnitski's invariant \cite[(1.12)]{Pu01FA}
\begin{gather} \label{F: Push formula}
  \xi(\lambda) = - \frac 1{2\pi } \int_0^{2\pi} \mu(\theta, \lambda)\,d\theta.
\end{gather}
It turns out that the path $r \in [0,1] \mapsto S(\lambda;H_r,H_0)$ is continuous, so that one can
introduce the "absolutely continuous part" $\mua(\theta,\lambda)$ of Pushnitski's invariant
as spectral flow of eigenvalues of $S(\lambda;H_r,H_0)$ through $e^{i\theta}.$ Under certain assumptions,
which are natural for stationary approach to the mathematical scattering theory, and under which operators
admit embedded eigenvalues, it is shown that
\begin{gather*}
  \xia(\lambda) = - \frac 1{2\pi } \int_0^{2\pi} \mua(\theta, \lambda)\,d\theta,
\end{gather*}
and that the difference $\mus(\theta, \lambda) := \mu(\theta, \lambda) - \mua(\theta, \lambda)$
does not depend on $\theta$ and is equal to $-\xis.$ This implies that $\xis$ takes integer values.
This result can be interpreted as a jump by an integer multiple of $2\pi$
of one of the scattering phases $\theta_j(\lambda,r),$
when $r$ crosses the "resonance" point $r_0,$ i.e. the point, for which the equation $H_r \psi = \lambda \psi, \ \lambda \in \Lambda,$
has an $L^2$ solution, in accordance with Pushnitski's formula (\ref{F: Push formula}).
This also agrees with a physical fact that one of the scattering phases $\delta_l(E)$ ($l$ is the angular quantum number)
makes a smooth jump by $\pi,$ when the energy of the incident particle crosses the resonance value $E_0$
\cite[XVIII.6]{Bohm}, where the jump is smooth since the lifetime of the resonance is finite.
It seems to be likely that $\xis$ is always an integer-valued function.

Results of this note, as well as many proofs, were inspired by the ideas of \cite{Pu01FA}.

\section{Results}

We will repeatedly use without mentioning the continuous dependence of isolated eigenvalues \cite[IV.3.5]{Kato}.

If $H_0$ is a self-adjoint operator and if $V$ is a bounded self-adjoint operator, then
let \cite[(4.1)]{Pu01FA}
$$
  M(z,r) := M(z; H_r, H_0) = \brs{H_r - \bar z}\brs{H_r - z}^{-1} \brs{H_0 - z}\brs{H_0 - \bar z}^{-1},
$$
where $z \in \mbC\setminus \mbR,$ and $H_r = H_0 + rV, \ r \in \mbR.$
Let for some $p \in [1,\infty]$
\begin{equation} \label{F: M(z,1) in Lp}
  M(z,1) - 1 \in \LpH{p}.
\end{equation}
\begin{lemma} If (\ref{F: M(z,1) in Lp}) holds, then for all $r \in [0,1]$ $M(z,r) - 1 \in \LpH{p}.$
\end{lemma}
\begin{proof} Since the inclusion $M(z,r) - 1 \in \LpH{p}$ is equivalent to
$(H_r - z)^{-1} - (H_0 - z)^{-1} \in \LpH{p}$ \cite[\S 4.1]{Pu01FA}, it is enough to prove this last inclusion.
Using the second resolvent identity, one has
\begin{equation*} \label{F: M(z,r) in Lp}
 \begin{split}
  (H_r - & z)^{-1} - (H_0 - z)^{-1} = - (H_r - z)^{-1} rV (H_0 - z)^{-1}\\
   & = - (H_r - z)^{-1}(H_1 - z)(H_1 - z)^{-1} rV (H_0 - z)^{-1} \\
   & = r(1+(1-r)V(H_1 - z)^{-1}) ((H_1 - z)^{-1} - (H_0 - z)^{-1}) \in \LpH{p},
 \end{split}
\end{equation*}
where the last operator belongs to $\LpH{p},$ since $(H_1 - z)^{-1} - (H_0 - z)^{-1} \in \LpH{p}$ by~(\ref{F: M(z,1) in Lp}).
\end{proof}

\begin{lemma} \label{L: Tr(V Rz(Hr)) is continuous}
If $VR_{z_0}(H_0) \in \LpH{p}$
for some $z_0 \in \mbC \setminus \mbR,$
then the function $(z,r) \in \mbC \setminus \mbR \times \mbR \mapsto VR_z(H_r)$
takes values in $\LpH{p}$ and $\LpH{p}$\tire continuous,
where $H_r = H_0 + rV.$
\end{lemma}
\begin{proof} The first and the second resolvent identities imply
$$
  VR_z(H_r) = V R_{z_0}(H_0)[1+(z-z_0)R_{z}(H_0)](1-rVR_z(H_r)) \in \LpH{p}.
$$
This formula also implies continuity of $VR_z(H_r).$
\end{proof}

We note that if $H$ is semibounded then $VR_z(H) \in \LpH{p}$ implies that $G(1+\abs{H})^{1/2} \in \LpH{2p}.$
Plainly, the assumption $VR_z(H) \in \LpH{1}$ is stronger than the trace
compatibility assumption $V\phi(H) \in \LpH{1}, \ \phi \in C_c(\mbR).$

From now on we will assume that for any $H \in \clA$ and $V \in \clA_0$
\begin{equation} \label{F: VR in trace class}
  VR_z(H) \in \LpH{p} \ \ \text{and} \ \ G(1+\abs{H})^{1/2} \in \LpH{2p}.
\end{equation}

\begin{lemma} \label{L: M(z) is continuous}
The function
$
  (z,r) \in \mbC_+ \times [0,1] \mapsto M(z,r) \in 1+\LpH{p}
$
is continuous in $\LpH{p}$\tire norm.
\end{lemma}
\begin{proof}
This follows from \cite[(4.22)]{Pu01FA}
\begin{multline*}
   M(z,r) - 1 = r(\bar z - z) (G(H_0-\bar z)^{-1})^* \\
   \times (J^{-1}+rT(z))^{-1}(G(H_0-z)^{-1})(1-(z - \bar z)(H_0-\bar z)^{-1})
\end{multline*}
and from the argument of the proof of \cite[Proposition 4.1 (iii)]{Pu01FA}.
\end{proof}
\begin{lemma} \cite[Lemma 4.1]{Pu01FA}
\label{L: norm of M to 0 uniformly} When $y \to +\infty$
$$
  \norm{M(\lambda+iy,H_r,H_0) - 1}_p \to 0
$$
uniformly with respect to $r \in [0,1].$
\end{lemma}
\begin{proof}
Consider the function \cite[\S 4.3]{Pu01FA}
\begin{equation} \label{F: S(z,r)}
  S(z,r) = 1 - 2i r B_0^{1/2}(z)J (1 + r T_0(z)J)^{-1} B_0^{1/2}(z),
\end{equation}
where $B_0(z) = \Im T_0(z).$
Repeating the argument of the proof of \cite[Lemma 4.1]{Pu01FA},
one can see that
$\norm{S(z,r) - 1}_p \to 0$ as $\Im z \to +\infty$
uniformly with respect to $r \in [0,1].$ Now, \cite[Theorem 4.1]{Pu01FA} completes the proof.
\end{proof}


Let $\clA$ be a trace compatible space. Let $p, \tilde p \in [1,\infty],$ $\tilde p \leq p.$
\begin{assump}\label{A: assumption 3}
Let $H_0 \in \clA,$ $V \in \clA_0$ and let $H_r = H_0+rV,$ $r \in [0,1].$
\\ (i) the condition (\ref{F: VR in trace class}) holds;
\\ (ii) there exists an open set $\Lambda = \Lambda(H_0,V) \subset \LambdaA,$ such that
$\LambdaA \setminus \Lambda$ is a discrete in $\LambdaA$ set and
for all $\lambda \in \Lambda$ the function $T_0(\lambda+iy)$ has
non-tangential limit values $T_0(\lambda \pm i0)$ in $\LpH{\infty};$
\\ (iii) for all $\lambda \in \Lambda$ $B_0(\lambda \pm i0) := \Im T_0(\lambda \pm i0) \in \LpH{\tilde p};$
\\ (iv) if $J^{-1}+T_0(\lambda+i0)$ is not invertible for $\lambda \in \Lambda$
then $\lambda$ is an eigenvalue of $H_0+V.$
\end{assump}

Assumption (ii) implies that for $\lambda \in \Lambda(H_0,V),$
$T_r(\lambda+i0)$ exists outside a discrete set in $[0,1]$ and $T_r(\lambda+i0) \in \LpH{\infty}.$
Indeed, since $T_0(\lambda+i0) \in \LpH{\infty},$
the operator $1+rJT_0(\lambda+i0)$ is invertible on a set with discrete complement. Hence, \cite[Lemma 1.4]{Az} completes the proof.

Let
$$
  \gamma = \gamma(H_0,V) := \set{(\lambda,r) \in \LambdaA \times [0,1] \colon (\lambda,r) \notin \Lambda(H_0,rV)}.
$$
This set will be called a \emph{resonance set}.

We need the following variant of the stationary formula for the scattering matrix,
see e.g. \cite{BYa92AA2}.
\begin{thm} \label{T: stationary rep-n for SM: 2}
Let $r \in [0,1].$
If $H_0 \in \clA,$ $V \in \clA_0,$ $H_r = H_0 + rV,$ and if Assumption \ref{A: assumption 3}
holds, then for all $\lambda \in \Lambda(H_0,rV)$
the scattering matrix $S(\lambda; H_r,H_0)$ exists,
and the stationary representation for the scattering matrix 
\begin{gather} \label{F: stationary rep-n for SM: 2}
  S(\lambda; H_r,H_0) = 1_\lambda - 2\pi i r Z_0(\lambda) J (1 + r T_0(\lambda+i0)J)^{-1}Z_0^*(\lambda)
\end{gather}
holds.  Moreover, $S(\lambda; H_r,H_0) - 1_\lambda \in \frS_{\tilde p}(\hilb_\lambda)$
for all $\lambda \in \Lambda(H_0,rV).$
\end{thm}

\begin{lemma} \label{L: S(r) is meromorphic}
The function $r \mapsto S(\lambda; H_r, H_0)$ is a meromorphic function,
which has analytic continuation to all real poles of $(J^{-1}+rT_0(\lambda+i0))^{-1}.$
\end{lemma}
\begin{proof}
Since $T_0$ is compact, the function $r \mapsto S(\lambda; H_r, H_0)$
is meromorphic by Theorem \ref{T: stationary rep-n for SM: 2}
and the analytic Fredholm alternative. Since it is also bounded (unitary-valued)
along $r \in \Lambda,$ it has analytic continuation to any real pole $r = r_0$
of $(J^{-1}+rT_0(\lambda))^{-1}.$
\end{proof}


It follows from \cite[(17)]{Az} 
that the function $r \mapsto w_+(\lambda; H_0, H_r)\Pi_{H_r}(\dot H_r)(\lambda)w_+(\lambda; H_r, H_0)$
also has analytic continuation to~\mbox{$r = r_0.$}


Let $\Gamma$ be a piecewise linear path in $\clA.$ We denote by $\Lambda(\Gamma)$
the intersection of the sets $\Lambda(H_j,H_j-H_{j-1}), \ j=1,2,\ldots,$ where $H_0,H_1,H_2,\ldots$ are the
nodes of~$\Gamma.$

\begin{thm} \label{T: S = Texp for res.case}
Let Assumption \ref{A: assumption 3} holds
and let $H_0, H_1 \in \clA.$
If $\Gamma = \set{H_r}_{r \in [0,1]}$ is a piecewise linear path in $\clA,$
connecting $H_0$ and $H_1,$ then for any $\lambda \in \Lambda(\Gamma)$
 \begin{gather} \label{F: S = T exp... res.case}
   S(\lambda; H_1, H_0) = \Texp\brs{-2\pi i \int_0^1 w_+(\lambda; H_0, H_r)\Pi_{H_r}(\dot H_r)(\lambda)w_+(\lambda; H_r, H_0)\,dr}.
 \end{gather}
holds.
\end{thm}
\begin{proof} 
For the straight line $\set{H_r}_{r \in [0,1]}$
(\ref{F: S = T exp... res.case}) follows from Lemma \ref{L: S(r) is meromorphic}
and (the proof of) \cite[Proposition 2.4]{Az}. For a piecewise linear path $\set{H_r}_{r \in [0,1]}$
the proof is the same as that of Theorem 
\cite[Theorem 2.5]{Az}.
\end{proof}

\begin{thm} \label{T: det S = exp(xi ac) for resonance case}
If $H_0,H_1 \in \clA$ and Assumption \ref{A: assumption 3} holds, then
\begin{gather*} \label{F: Birman Krein res.case}
  -2\pi i \xia_{H_1,H_0}(\lambda) = \ln \det S(\lambda; H_1, H_0), \quad a.\,e.\, \ \lambda \in \LambdaA,
\end{gather*}
where the branch of the logarithm is chosen in such a way, that the function
$r \in [0,1] \mapsto \ln \det S(\lambda; H_r,H_0)$ is continuous. Here $\set{H_r}$
is a piecewise linear path.
\end{thm}
The proof of this theorem is the same as that of Theorem \cite[Theorem 2.7]{Az}. 

\begin{cor} \label{C: def of xia is good for res case}
  The definition of $\xia$ and $\xis$ does not depend on the choice of the piecewise linear path $\Gamma.$
\end{cor}



We denote by $e^{i\theta_j(z,r)}, \ j = 1,2,\ldots,$ the eigenvalues of the operator $M(z;H_r,H_0).$

\begin{lemma} \label{L: S(lambda+i0) exists}
If $(\lambda,r) \notin \gamma(H_0,V),$
then the limits values $e^{i\theta_j(\lambda+i0,r)}$ exist.
\end{lemma}
\begin{proof}
By \cite[Theorem 4.1]{Pu01FA}, for $z \in \mbC \setminus \mbR$ the eigenvalues of $M(z;H_r,H_0)$ and $S(z,r)$ coincide
(see (\ref{F: S(z,r)})).
By Assumption \ref{A: assumption 3}(ii), the norm limit $B_0(\lambda+i0)$ exists.
Since $(\lambda,r) \notin \gamma,$
the norm limit $(J^{-1} + rT_0(\lambda+i0))^{-1}$ exists by \cite[Lemma 1.4]{Az}. 
It follows that the norm limit $S(\lambda+i0,r)$ exists. Hence,
the limit values $e^{i\theta_j(\lambda+i0,r)}$ of the eigenvalues of $S(z,r)$ also exist.
\end{proof}

In case $y = \Im z \to \infty$ the limit values of $e^{i\theta_j(\lambda+iy,r)}$
are equal to~$1$ \cite[Lemma 4.1]{Pu01FA}.
The functions $r \mapsto \theta_j(\lambda+iy,r)$ can be chosen to be continuous and such that $\theta_j(\lambda+iy,0)=0.$
We denote by $\theta_j(\lambda+i0,r)$ the limit $\lim_{y \to 0^+} \theta_j(\lambda+iy,r).$
This allows one to define Pushnitski's $\mu$\tire invariant \cite{Pu01FA} by formula
$$
  \mu(\theta,\lambda; H_1, H_0) = - \sum_{j=1}^\infty \SqBrs{\frac{\theta - \theta_j(\lambda+i0,1)}{2\pi}},
$$
where the path to $\theta_j(\lambda+i0,1)$ is taken along $y$ as $y$ goes from $+\infty$ to $0,$
and $[x] = \max \set{n \in \mbZ \colon n \leq x}.$
We also introduce the "absolutely continuous" part of Pushnitski's $\mu$\tire invariant by the same formula
$$
  \mua(\theta,\lambda; H_1, H_0) = - \sum_{j=1}^\infty \SqBrs{\frac{\theta - \theta^*_j(\lambda,1)}{2\pi}},
$$
where $e^{i\theta^*_j(\lambda,r)}$ are eigenvalues of $S(\lambda; H_r, H_0).$
In other words, $\mua(\theta,\lambda)$ (respectively, $\mu(\theta,\lambda)$)
is the spectral flow of eigenvalues $e^{i\theta^*_j(\lambda,r)}, \ j = 1,2,\ldots,$ (respectively, $e^{i\theta_j(\lambda+i0,r)}$)
of $S(\lambda; H_r,H_0)$ (respectively, $M(\lambda+i0; H_r,H_0)$)
through $e^{i\theta}$ in anticlockwise direction, as $r$ moves from $0$ to $1.$

\begin{prop} \label{P: xia = int mua} For a.e. $\lambda \in \Lambda$
$$
  \xia(\lambda; H_1,H_0) = - \frac 1{2\pi} \int_0^{2\pi} \mua(\theta,\lambda; H_1,H_0)\,d\theta.
$$
\end{prop}
\begin{proof} This follows from
Lemma \ref{L: S(r) is meromorphic} and Theorem \ref{T: det S = exp(xi ac) for resonance case}.
\end{proof}

%
%

If $(\lambda,r_0) \in \gamma,$
then we denote by $\theta^\pm_j(\lambda,r_0)$ the left and the right limits of
$\theta_j(\lambda,r_0)$ as $r \to r_0,$ provided that these limits exist.

\begin{lemma} \label{L: theta+ - theta- is 2pi}
If $(\lambda_0,r_0) \in \gamma$ is a resonance point,
then for any $j=1,2,\ldots$ the limits $\theta^\pm_j(\lambda_0,r_0)$ exist and
their difference $\theta^+_j(\lambda_0,r_0) - \theta^-_j(\lambda_0,r_0)$
is an integer multiple of $2\pi.$
\end{lemma}
\begin{proof}
If $(\lambda,r) \notin \gamma,$ then by (the proof of) Lemma \ref{L: S(lambda+i0) exists}
the limit function $S(\lambda+i0, r)$ exists.
Using $\dim \ker(AB-\lambda) = \dim \ker(BA-\lambda)$ (see e.g. \cite[(4.24)]{Pu01FA})
and $\pi Z_r^*(\lambda)Z_r(\lambda) = B_r(\lambda + i0)$ (see e.g. \cite[\S 9]{Pu01FA})
it follows from Theorem \ref{T: stationary rep-n for SM: 2} and (\ref{F: S(z,r)})
that the eigenvalues of $S(\lambda+i0, r)$ and that of
$S(\lambda; H_r,H_0)$ coincide outside the resonance set $\gamma.$
Since by Lemma \ref{L: S(r) is meromorphic} the function $r \mapsto S(\lambda; H_r,H_0)$
is continuous, its eigenvalues $e^{i\theta^*_j(\lambda,r)}$
are continuous functions of~$r.$
This implies that $\theta^\pm_j(\lambda_0,r_0)$ exist and
$\theta^+_j(\lambda_0,r_0) - \theta^-_j(\lambda_0,r_0)$ is an integer multiple of~$2\pi.$
\end{proof}

\begin{lemma} \label{L: mu = mua+mus} The difference
$$
  \mus(\theta,\lambda; H_1,H_0) := \mu(\theta,\lambda;H_1,H_0) - \mua(\theta,\lambda;H_1,H_0)
$$
is an integer-valued function, which does not depend on $\theta.$
The function $r \mapsto \mus(\lambda; H_r,H_0)$ is constant outside of the resonance set
$\gamma_\lambda:=\set{r \in [0,1] \colon (\lambda,r) \in \gamma},$
while at any point $r_0$ of $\gamma_\lambda$ its jump is equal to
\begin{equation} \label{F: sum of theta+ - theta-}
  \frac 1{2\pi} \sum_{j=1}^\infty (\theta^+_j(\lambda,r_0) - \theta^-_j(\lambda,r_0)).
\end{equation}
\end{lemma}
\begin{proof}
Let $\tau_y = \tau_y(\lambda),$ $y \in (0,\infty],$ be a path in $\set{(z,r) \in \mbC_+\times [0,1]},$
consisting of two straight lines $(\lambda +i0,1) \to (\lambda +iy,1)$
and $(\lambda +iy,1) \to (\lambda +iy,0),$ if $y<\infty,$ and of one straight line $(\lambda +i0,1) \to (\lambda +i\infty,1),$
if $y = \infty.$ We will identify these paths with their images in the set of unitary operators under the map $M(z,r).$
Any two paths $\tau_y$ and $\tau_\eps,$
connecting the unitary operator $M(\lambda+i0,1)$ with
the identity operator $1$ are homotopic.
Since by Lemma \ref{L: M(z) is continuous} the function $M(z,r)$ is continuous,
the $\mu$\tire invariant (i.e. the spectral flow of eigenvalues on the unit circle), computed along these two
paths, coincide. Lemma \ref{L: norm of M to 0 uniformly}
implies that the spectral flow along $\tau_\infty$ is also the same.
Hence, after letting $\eps \to 0,$ Lemma \ref{L: theta+ - theta- is 2pi} implies that $\mus$ does not depend on $\theta.$
The sum in (\ref{F: sum of theta+ - theta-}) is finite, since $\theta^+_j(\lambda,r_0) - \theta^-_j(\lambda,r_0)$
is a multiple of $2\pi$ and $\theta_j^\pm(\lambda,r_0) \to 0$ as $j \to \infty.$
\end{proof}
It follows that
\begin{equation} \label{F: mu sing = sum gamma lambda}
  \mus(\lambda; H_1,H_0) = \frac 1{2\pi} \sum_{r \in \gamma_\lambda} \sum_{j=1}^\infty (\theta^+_j(\lambda,r) - \theta^-_j(\lambda,r)).
\end{equation}
Since $\gamma_\lambda \subset [0,1]$ consists of real poles of a meromorphic function,
the first sum in this formula is also finite.



\begin{lemma} \label{L: xi = lim arg D(y)}
If Assumption \ref{A: assumption 3} holds with $p=1,$ then
the following equality holds true for a.e. $\lambda$
$$
  \xi(\lambda) = - \frac 1{2\pi} \int_0^{2\pi} \mu(\theta,\lambda)\,d\theta.
$$
\end{lemma}
\begin{proof}
It follows from \cite[(4.4)]{Pu01FA} that in $\LpH{p}$
$$
  \frac d{dr} M(z; H_r, H_s)\Big|_{r=s} = -2iy R_z(H_s) V R_{\bar z} (H_s),
$$
where $y = \Im z.$
This equality and the multiplicative property $M(z; H_{2}, H_{0}) = M(z; H_{2}, H_{1})M(z; H_{1}, H_{0})$
imply that
$$
  \frac d{dr} M(z; H_r, H_s) = -2iy R_z(H_r) V R_{\bar z} (H_r) M(z; H_r, H_s).
$$
Hence, \cite[Lemma A.1]{Az} 
implies that
$$
  M(z; H_r, H_{0}) = \Texp\brs{-2iy \int_{0}^r R_z(H_s) V R_{\bar z} (H_s)\,ds}.
$$
Since $p=1,$ by 
\cite[Lemma A.3]{Az}
\begin{equation*}
  \begin{split}
    i \arg \det M(z; H_r,H_{0}) & = -2iy \int_{0}^r \Tr(R_z(H_s) V R_{\bar z}(H_s)) \,ds.
    \\ & = \int_{0}^r \Tr(V (R_{\bar z}(H_s) - R_z(H_s)) \,ds,
  \end{split}
\end{equation*}
where the value of $\arg$ is chosen such that the left hand side is continuous with respect to $r.$
Hence, for $\phi \in C_c^\infty(\Lambda(H_0,V)),$ we have
\begin{multline}
  i \int_\Lambda  \arg \det M(\lambda+iy; H_r,H_0) \phi(\lambda)\,d\lambda
   \\ = \int_\Lambda \brs{\int_{0}^r \Tr(V (R_{\lambda-iy}(H_s) - R_{\lambda+iy}(H_s)))\phi(\lambda) \,ds}\,d\lambda.
\end{multline}

If we take the limit $\lim_{y \to 0^+}$ then on the left hand side we can interchange
the limit and the integral, since $y \mapsto \det M(\lambda+iy; H_r,H_s)$ is continuous up to the cut along $\Lambda.$
The expression under the second integral is continuous by Lemma \ref{L: Tr(V Rz(Hr)) is continuous}.
Hence, we can interchange the integrals in it:
$$
  \int_{0}^r \brs{\int_\Lambda \Tr(V (R_{\lambda-iy}(H_s) - R_{\lambda+iy}(H_s)))\phi(\lambda) \,d\lambda}\,ds.
$$
By the same reason, we can interchange the $\lambda$\tire integral and the trace to get
$$
  \int_{0}^r \brs{\Tr\SqBrs{\int_\Lambda V (R_{\lambda-iy}(H_s) - R_{\lambda+iy}(H_s))\phi(\lambda) \,d\lambda}}\,ds.
$$

By Stone's formula (see e.g. \cite{RS1}) and a simple approximation argument,
the inner integral converges (in $so$\tire topology) to $- 2 \pi i V\phi(H_s)$
as $y \to 0^+.$
The convergence in $\LpH{1}$\tire topology can be shown in the following way.
Write $V = VE^{H_s}_\Delta + VE^{H_s}_{\mbR \setminus \Delta}.$
Here $\Delta$ is a big enough segment, containing $\supp(\phi).$ For $VE^{H_s}_\Delta$ the convergence
in $\LpH{1}$ follows from e.g. \cite[(2.2)]{Pu01FA}.
For $VE^{H_s}_{\mbR \setminus \Delta}$ the convergence follows from the monotone decreasing to $0$ of the family
$i \sqrt{\abs{V}}E^{H_s}_{\mbR \setminus \Delta}[R_{\lambda-iy}(H_s) - R_{\lambda+iy}(H_s)]\sqrt{\abs{V}} \in \LpH{1}$ for small enough values of $y.$

Hence, since the integral converges in $\LpH{1}$ topology, we can interchange the limit
and the trace. Finally, we can interchange the $r$\tire integral and the limit by the dominated convergence theorem.

Hence,
$$
  - \int_\Lambda  \arg \det M(\lambda+i0; H_1,H_{0}) \phi(\lambda)\,d\lambda
   = 2\pi \int_{0}^1 \Tr\brs{V\phi(H_s)}\,ds.
$$
The right hand side is equal to $2\pi \xi_{H_1,H_0}(\phi).$ Since $\xi$ is absolutely continuous
\cite[Theorem 2.9]{AS2},
it follows that for a.e. $\lambda$
$$
  2\pi \xi(\lambda) = - \arg \det M(\lambda+i0; H_1,H_0).
$$
It follows that
$$
  \xi(\lambda) = - \frac 1{2\pi} \sum_{j=1}^\infty \theta_j(\lambda+i0,1; H_1,H_0) =
  - \frac 1{2\pi} \int_0^{2\pi} \mu(\theta,\lambda)\,d\theta.
$$
\end{proof}

{\bf Remark.} This lemma together with the argument of the proof of \cite[Theorem 6.1]{Pu01FA}
implies that
\begin{equation} \label{F: xi = lim arg D}
   \xi(\lambda) = \lim_{y \to 0^+} \arg D_{H/H_0}(\lambda+iy),
\end{equation}
where $D_{H/H_0}(z) = \det(1+JT_0(z)),$ which for trace class perturbations was
proved by \KreinMG\ in \cite{Kr53MS}. Conversely, if (\ref{F: xi = lim arg D}) was known
for perturbations $V$ satisfying $VR_z(H) \in \LpH{1},$ then combined with \cite[Theorem 6.1]{Pu01FA}
it would imply Lemma \ref{L: xi = lim arg D(y)}.

\begin{thm} If Assumption \ref{A: assumption 3} holds with $p=1,$ then for a.e. $\lambda \in \mbR$
$\xis(\lambda) = -\mus(\lambda),$ and hence, $\xis$ is an integer-valued function.
\end{thm}
\begin{proof} This follows from Lemmas \ref{L: mu = mua+mus},
\ref{L: xi = lim arg D(y)} and Proposition \ref{P: xia = int mua}.
\end{proof}
It follows that $-\xis(\lambda)$ is equal to the right hand side of (\ref{F: mu sing = sum gamma lambda}).

\begin{cor} If the Assumption \ref{A: assumption 3} holds with $p=1,$ then for a.e. $\lambda \in \mbR$
$$
  \det S(\lambda; H_1,H_0) = e^{-2\pi i \xi(\lambda)}.
$$
\end{cor}

It follows directly from the definition of $\xis$ that
if $V \geq 0,$ then the function $\xis$ is non-negative.
The function $\xis$ is non-zero on $\LambdaA,$
if there is at least one moving embedded eigenvalue.
The last is not necessary: embedded eigenvalues, present at $r=0,$ can just stay constant,
or, what seems to be more likely, disappear for $r>0.$
In this case, $\xis$ is zero outside of $\set{\lambda_j}.$
Absolute continuity of $\xis$ implies that $\xis$ is zero on $\LambdaA.$


Though disappearance or stability of embedded eigenvalues may seem to be unlikely,
the consideration of one-dimensional short range Schr\"odinger operators with embedded eigenvalues shows,
that it is plausible (Appendix \ref{S: appendix example}).
Indeed, in order to ensure square summability of eigenfunction, the barriers and pits of the potential
and the energy of the eigenfunction should be finely tuned, see e.g. \cite[Chapter 4]{EK} or Appendix A.
A slight change of potential or of boundary condition destroys square summability of the eigenfunction.

Note that examples of Schr\"odinger operators with embedded eigenvalues have artificial potentials.
At the same time, some natural examples of magnetic Schr\"odinger operators
$(-i\nabla-a)^2+V$
have stable and moving (as $r$ changes)
embedded eigenvalues, see \cite{AHS,BBR,ABBFR}. One can look for examples with non-zero $\xis$
on the absolutely continuous spectrum in this direction.


\appendix

\section{Example} \label{S: appendix example}
\subsection{An example of short range Schr\"odinger operator with embedded eigenvalues}

We construct a short range Schr\"odinger operator (in the sense of \cite[Definition 3.1]{Agm})
on a half-line $(0,\infty)$
with an embedded eigenvalue,
using an idea from \cite{EK}, which goes back to \cite{vNW}.
Other examples of Schr\"odinger operators with embedded eigenvalues can be found in \cite[Chapter 4]{EK},
but those examples are not short range.

Let $0 < \lambda < w,$ let $a_0 < b_0 < a_1 < b_1 < \ldots$ and let
$$
  W(x) = \left\{
    \begin{array}{cl}
      w, & \text{if} \ x \in [a_n,b_n] \ \ \text{for some $n$} \\
      0, & \text{if otherwise.}
    \end{array}
    \right.
$$
If for some $s > \frac 12$ \ $b_n-a_n = n^{-s},$ then $W$
satisfies the estimate \cite[(1.3)]{Agm},
and, consequently, $W(x)$ is a short range potential \cite{Schech}, \cite[\S 3, Remark 2]{Agm}.

If $a_{n+1}-b_n, \ n=0,1,2\ldots,$ are chosen to be equal to the wavelength of a free particle with energy $\lambda,$
then the eigenfunction $\psi$ enters the interval $(b_n,a_{n+1})$
and leaves it with the same phase.
If $a_n$ is chosen such that $\frac {\psi'(a_n)}{\psi(a_n)} = -\sqrt{w-\lambda}$ then $\psi$ enters and leaves the barrier
$[a_n,b_n]$ with the same phase.
The boundary condition and the first barrier are chosen
so that $\psi$ enters the first barrier with $\frac {\psi'(a_0)}{\psi(a_0)} = -\sqrt{w-\lambda}.$
The amplitude of $\psi$ exponentially decreases on $[a_n,b_n]$ and is constant on $(b_n,a_{n+1}).$
If $s = \frac 34,$ then the sum of the lengthes of $[a_n,b_n]$ will be $\infty,$ and
it follows that $\psi$ decreases as $e^{- c \sqrt[4]{x}},$ $c>0.$

One can construct similarly a short range Schr\"odinger operator with an embedded eigenvalue on the whole $\mbR,$
e.g. by taking the boundary value $\psi(0)=1,$ $\psi'(0)=0$ and reflecting the potential $W(x)$ above.

\begin{lemma} \label{L: example of SR potential with e.e.} $W$ is a short range potential.
  The operator $-\frac{d^2}{dx^2} + W$ has an embedded eigenvalue.
\end{lemma}


\subsection{An example of trace compatible space with embedded eigenvalues which satisfies Assumption~\ref{A: assumption 3}}
We recall the definitions of the weighted Hilbert $L^{2,s}(\mbR^\dm)$
and Sobolev $\euH_{m,s}(\mbR^\dm)$ spaces from \cite{Agm}. By definition, for $s \in \mbR$ and $m =0,1,2,\ldots$
$$
  L^{2,s}(\mbR^\dm) = \set{u(x) \colon (1+{x}^2)^{s/2} u(x) \in L^2(\mbR^\dm)},
$$
and
$$
  \euH_{m,s}(\mbR^\dm) = \set{u(x) \colon D^\alpha u \in L^{2,s}(\mbR^\dm), \ 0 \leq \abs{\alpha} \leq m}.
$$
Different variants of the following lemma are well-known (see e.g. \cite[Lemma 4.7.8]{Ya} and the proof of \cite[Theorem 4.2]{Agm}).
But in \cite[Theorem 4.2]{Agm} (modification of) the operator $1+JT_0(\lambda+i0)$
acts in the Hilbert space $\euH_{2,-s}$ for some $s > \frac 12,$
while \cite[Lemma 4.7.8]{Ya} requires checking the strong smoothness of perturbation.
So, we give the proof of this lemma for completeness.
\begin{lemma} \label{L: if 1+JT0 is invertible ...} 
  Let $\clA_0 = \set{V \in L^\infty(\mbR^\dm) \colon \exists\, C>0 \ \exists\, s > \frac 12, \ \abs{V(x)} \leq C (1+x^2)^{-s/2}}.$
  Let $H_0 = H_{00}+W$ where $W$ is a short range potential and $H_{00} = -\Delta$ with $\dom(H_{00}) = \euH_2(\mbR^\dm),$
  and let $V = GJG \in \clA_0,$ where $G = \abs{V}^{1/2}.$
  If the operator $1+rJT_0(\lambda+i0), \ \lambda > 0,$ is not invertible for some $r$ then
  the equation $H_r \psi = \lambda \psi$ has an $L^2$ solution. Moreover this solution decreases faster than
  $(1+x^2)^{-s/2}$ for any $s \in \mbR.$
\end{lemma}
\begin{proof} One can assume that $r = 1.$
%
Since $T_0(\lambda+i0)$ is compact, by Fredholm alternative the operator $1+rJT_0(\lambda+i0)$ is not invertible
if and only if  there exists $u \in L^2(\mbR^\dm)$ such that $JG R_{\lambda+i0}(H_0)G^* u = -u.$
Since for any $s > \frac 12$ $G^* u \in L^{2,s}(\mbR^\dm),$ by the limiting absorption principle
\cite[Theorem 4.2]{Agm} it follows that
$$\psi:=R_{\lambda+i0}(H_0)G^* u \in \euH_{2,-s}(\mbR^\dm).$$

By definition of a short range potential \cite[Definition 3.1]{Agm} (see also the remark after this definition)
$W R_{\lambda+i0}(H_0)G^*u \in L^{2,-s+1+\eps}(\mbR^\dm).$ So, choosing $s>\frac 12$ so that $s < \frac 12 + \eps,$
the last equality, combined with the equality
$$
  R_{\lambda+i0}(H_0) = R_{\lambda+i0}(H_{00})(1+W R_{\lambda+i0}(H_0)),
$$
imply that $\psi$ is a $\sqrt \lambda$\tire outgoing function (\cite[Definition 4.1]{Agm}).
By \cite[Theorem 4.2]{Agm}, $\psi$ is a solution of the equation $(-\Delta + W - \lambda)\psi = G^*u.$
Since $G^*u = -V\psi,$ $\psi$ is a solution of the equation $(-\Delta + W +V)\psi = \lambda \psi.$
Since $\psi$ is a $\sqrt \lambda$\tire outgoing function, by \cite[Lemma 4.2]{Agm} $\psi \in \euH_{2,s}(\mbR^\dm)$
for any $s \in \mbR.$
\end{proof}
Concerning the inverse of this lemma, if $H_r \psi = \lambda \psi$ has a $L^{2,s}$ solution ($s > \frac 12$)
then one can show that $1+rJT_0(\lambda+i0)$ is not invertible.

We note that by \cite[Theorem 4.3]{Agm}, for any pair of operators from the affine space $-\Delta + W + \clA_0$
the resonance set $\gamma$ is closed.


\begin{prop} If $W$ is a short range potential from Lemma \ref{L: example of SR potential with e.e.},
then for $\dm=1$ the affine space $-\Delta + W + \clA_0$ satisfies Assumption \ref{A: assumption 3}.
\end{prop}
\begin{proof} It is known that the condition (ii) of the Assumption holds, see e.g. \cite[\S 2]{BYa92AA2}, \cite{Agm}.
The condition (iv) holds by Lemma \ref{L: if 1+JT0 is invertible ...}.
\cite[Theorem B.9.2]{Si82BAMS} implies that the condition (i) and (iii) hold with $p = 1.$
\end{proof}

%
\mathsurround 0pt
\ndef{\AndSoOn}{$\dots$}

\end{document}